\documentclass[12pt,centertags,oneside]{amsart}

\usepackage{amscd,amsxtra,calc}
\usepackage{cmmib57}
\usepackage{url}

\usepackage{amscd}
\usepackage{pstricks}
\usepackage{color}
\usepackage[a4paper,width=16.2cm,top=3cm,bottom=3cm]{geometry}

\numberwithin{equation}{section}

\theoremstyle{plain}
\newtheorem{thm}{Theorem}[section]
\newtheorem{theorem}[thm]{Theorem}

\newtheorem{lemma}[thm]{Lemma}
\newtheorem{corollary}[thm]{Corollary}
\newtheorem{proposition}[thm]{Proposition}
\theoremstyle{definition}
\newtheorem{remark}[thm]{Remark}

\newtheorem{definition}[thm]{Definition}

\numberwithin{equation}{section}

\newcommand{\sO}{{\mathcal O}}

\newcommand{\C}{{\mathbb C}}

\newcommand{\F}{{\mathbb F}}
\newcommand{\G}{{\mathbb G}}

\renewcommand{\P}{{\mathbb P}}
\newcommand{\Q}{{\mathbb Q}}
\newcommand{\R}{{\mathbb R}}

\newcommand{\Z}{{\mathbb Z}}

\newcommand{\id}{{\rm id\hspace{.1ex}}}

\newcommand{\Aut}{{\rm Aut\hspace{.1ex}}}

\title [nonfinitely generated automorphism group]{A surface in odd characteristic with discrete and non-finitely generated automorphism group}

\author{Keiji Oguiso}
\address{Mathematical Sciences, the University of Tokyo, Meguro Komaba 3-8-1, Tokyo, Japan, and National Center for Theoretical Sciences, Mathematics Division, National Taiwan University, Taipei, Taiwan}
\email{oguiso@ms.u-tokyo.ac.jp}

\thanks{The author is supported by JSPS Grant-in-Aid (S) 15H05738, JSPS Grant-in-Aid (B) 15H03611, KIAS Scholar Program and by NCTS Scholar Program.}

\begin{document}

\maketitle
\begin{abstract}
It was proved by Tien-Cuong Dinh and me that there is a smooth complex projective surface whose automorphism group is discrete and not finitely generated. In this paper, after observing finite generation of the automorphism group of any smooth projective surface birational to any K3 surface over any algebraic closure of the prime field of odd characteristic, we will show that there is a smooth projective surface, birational to some K3 surface, such that the automorphism group is discrete and not finitely generated, over any algebraically closed field of odd characteristic of positive transcendental degree over the prime field.
\end{abstract}

\section{Introduction}

Let $p$ be an odd prime integer and let $\F_p := \Z/(p)$ be the prime field of characteristic $p$. Let $\F_p(t)$ be a purely transcendental extension of degree one of the field $\F_p$. We choose and fix an algebraic closure $k_0$ of $\F_p$ and an algebraically closed field $k$ such that $\F_p(t) \subset k$, eg. an algebraic closure of the field $\F_p(t)$. Note that any algebraically closed field of characteristic $p$ is isomorphic to either $k_0$ or some $k$ defined here. For our purpose, we may and do assume that 
$$\F_p \subset \F_{p^n} \subset k_0 \subset k\,\,$$ 
for all integers $n \ge 1$. Here $\F_{p^n}$ is a finite field of cardinality $p^n$. 

For a variety $V$ defined over a field $K$, we denote the group of the automorphisms of $V$ over $K$ by ${\rm Aut}\, (V/K)$ (See also Remark \ref{rem1}) and for a field extension $K \subset L$, we denote $V \times_{{\rm Spec}\, K} {\rm Spec}\, L$ by $V_L$.

Our main results are Theorem \ref{thm1} and 
Corollary \ref{cor1} below. Both (1) and (2) in Theorem \ref{thm1} are related to a question posed by \cite[Problem 1.2]{DO19}; (2) gives an affirmative answer in any odd characteristic, 
whereas (1) provides a negative evidence over $k_0$. 

\begin{theorem}\label{thm1}
\begin{enumerate}
\item Let $k_0$ be the base field. Then for any smooth projective surface $Y$ birational to a K3 surface over $k_0$ and for any field extension $k_0 \subset L$, the automorphism group ${\rm Aut}(Y_L/L)$ is finitely generated. 
\item Let $k$ be the base field. Then there is a smooth projective surface $Y$ birational to some K3 surface such that ${\rm Aut}\, (Y/k)$ is not finitely generated. 
\end{enumerate}
\end{theorem}

Theorem \ref{thm1} (1) is a special case of a slightly more general result on finite generation of the discrete automorphism group ${\rm Aut}(Y)/{\rm Aut}^0(Y)$ of a smooth projective surface $Y$ defined over $k_0$ (Theorem \ref{thm22} (3)). Theorem \ref{thm22} (3) is of its own interest and gives an answer to a question by the referee.  

\begin{corollary}\label{cor1}
Let $k$ be the base field. Then, for any integer $d$ such that $d \ge 2$, there is a smooth projective variety $Y_d$ of $\dim\, Y_d = d$ such that ${\rm Aut}(Y_d/k)$ is discrete and not finitely generated.
\end{corollary}

\begin{remark}\label{rem1} Let $K$ be an algebraically closed field. 
\begin{enumerate}
\item Let $V$ be a projective variety defined over $K$. 
Then the group ${\rm Aut}\, (V/K)$ has a natural scheme structure as a locally noetherian subscheme of the Hilbert scheme ${\rm Hilb}\, (V \times V)$ under the identification of an automorphism with its graph. We denote by ${\rm Aut}^0(V/K)$ the connected component containing $\id_V$. We say that ${\rm Aut}\, (V/K)$ is discrete if ${\rm Aut}^0(V/K)$ is reduced and ${\rm Aut}^0(V/K) = \{\id_V\}$. If $V$ is smooth, then $H^0(V, T_V)$ is the Zariski tangent space of ${\rm Aut}\, (V/K)$ at $\id_V$, therefore, ${\rm Aut}\, (V/K)$ is discrete if and only if $H^0(V, T_V) = 0$. 
\item  Let $S$ be a K3 surface defined over $K$, that is, a smooth projective surface defined over $K$ with $h^1(S,\sO_S) = 0$ and with a nowhere vanishing global regular $2$-form. Then $S_L$ is also a K3 surface over $L$ for any field extension $K \subset L$. Recall that $H^0(S, T_S) = 0$ also in positive characteristic by \cite[Theorem 7]{RS76}. Therefore, ${\rm Aut}\, (S/K)$ is discrete, i.e., ${\rm Aut}^0(S/K) = \{\id_S\}$. Recall also that ${\rm Aut}\, (S/K) = {\rm Bir}\, (S/K)$ by the minimality of the surface $S$. If we have a birational morphism $\tau : T \to S$ from a smooth projective surface $T$, then we have an inclusion $H^0(T, T_T) \subset H^0(S, T_S)$ via $\tau$ and therefore ${\rm Aut}\, (T/K)$ is also discrete as well. Moreover, ${\rm Aut}\, (T/K)$ can be regarded as a subgroup of ${\rm Aut}\, (S/K) = {\rm Bir}\, (S/K)$ via $\tau$ as follows: 
$${\rm Aut}\, (T/K) \subset {\rm Aut}\, (S/K)\,\, ;\,\, f \mapsto \tau \circ f \circ \tau^{-1}\,\, .$$
\end{enumerate}
\end{remark}

This work is much inspired by recent two remarkable works, due to Lesieutre \cite{Le17} in which a $6$-dimensional example as in Theorem \ref{thm1}(2), also over characteristic $2$, is constructed, and due to Dinh and me \cite{DO19} in which a complex surface example as in Theorem \ref{thm1}(2) is finally constructed. 

Let $S$ be a K3 surface defined over an algebraically closed field $K$. Sterk \cite{St85} shows the finite generation of ${\rm Aut}\, (S/K)$ when $K$ is of characteristic zero by using the Torelli theorem for complex K3 surfaces (See also Lemma \ref{lem21}). Then Lieblich and Maulik \cite[Thorem 6.1 and its proof]{LM18} shows the finite generation of ${\rm Aut}\, (S/K)$ when $K$ is of odd characteristic as Theorem \ref{thm2} below. They reduce to characteristic zero when $S$ is not supersingular (Theorem \ref{thm2} (2)), while they use the crystalline Torelli theorem, which is not yet settled in characteristic $2$, when $S$ is supersngular.

\begin{theorem}\label{thm2} 
Let $S$ be a K3 surface defined over an algebraically closed field $K$ of odd characteristic. Then 
\begin{enumerate}
\item ${\rm Aut}\, (S/K)$ is finitely generated.
\item Assume in addition that $S$ is not supersingular. Then there are a discrete valuation ring $R$ with residue field is $K$ and fraction field $Q(R)$ of characteristic $0$ and a smooth projective morphism $\pi : X \to {\rm Spec}\, R$ with special fiber $S$ such that the specialization map 
$${\rm Aut}\, (\tilde{S}/\tilde{K}) \to {\rm Aut}\, (S/K)$$ 
has finite kernel and cokernel. Here $\tilde{S}$ is the geometric generic fiber of $\pi$ and $\tilde{K}$ is an algebraic closure of the fractional field $Q(R)$, in particular, $\tilde{S}$ is a K3 surface defined over an algebraically closed field $\tilde{K}$ of charcateristic zero. 
\end{enumerate}
\end{theorem}

We prove Theorem \ref{thm1} (1) as an application of Theorem \ref{thm22} 
in Section \ref{sect2}. 

Our proof of Theorem \ref{thm1} (2) is quite close to \cite{DO19}. As in \cite{DO19}, we explicitly construct a desired surface $Y$ from some special Kummer K3 surface $X$ in odd characteristic. In Section \ref{sect3}, we define this surface $X$ and prove Theorem \ref{thm1} (2) by studying the surface $X$ and its suitable blow-up. Complex surfaces similar to $X$ are fully studied in \cite{Og89} and effectively applied in \cite{DO19}. However, some arguments there are based on the global Torelli theorem for complex K3 surfaces (see eg. \cite[Chapter VIII]{BHPV04}) which is not available over $k$. We also use a result due to Jang \cite[Proposition 3.5]{Ja13} on the finiteness of canonical representation of any non-supersingular K3 surface defined over any algebraically closed field of odd characteristic (Theorem \ref{thm32}). This substitutes the finiteness of canonical representation in characteristic $0$ (\cite[Theorem 14.10]{Ue75}) used in \cite{DO19}. 

We prove Theorem \ref{thm1} (2) in Section \ref{sect3} and Corollary \ref{cor1} in Section \ref{sect4}. 

Throughout this paper, for a variety $V$ defined over a field $K$ and for closed subsets $W_1$, $W_2$, $\ldots$, $W_n$ of $V$, we denote
$${\rm Aut}\, (V/K, W_1, W_2, \ldots, W_n) := \{ f \in {\rm Aut}\, (V/K)\, |\, f(W_i) = W_i\,\, (\forall i)\,\, \}\,\, .$$ 

\medskip\noindent
{\bf Acknowledgements.} I would like to thank Professors Tien-Cuong Dinh, Igor Dolgachev, Jun-Muk Hwang, H\'el\`ene Esnault, Yuya Matsumoto, Junichiro Noguchi, Takeshi Saito for valuable discussion and help. Especially, I would like to express my thanks to Professor Tien-Cuong Dinh for sharing many ideas since our previous joint work \cite{DO19} and his warm encouragement, Professor Jun-Muk Hwang for his invitation to one day workshop at KIAS which was very helpful to make the final version of this paper and Professor H\'el\`ene Esnault and the referee for many valuable comments most of which are effectively reflected in this paper.

\section{Proof of Theorem \ref{thm1} (1)}\label{sect2}

Our main result of this section is Theorem \ref{thm22} (3). We then deduce Theorem \ref{thm1} (1) as an application of Theorem \ref{thm22} (3) and Lemma \ref{lem21} below. 

As in \cite{DO19}, the following theorem will be frequently used in this paper.

\begin{thm}\label{thm21}
Let $G$ be a group and $H \subset G$ a subgroup of $G$. Assume that $H$ is of finite index, i.e., $[G :H] < \infty$. Then, the group $H$
 is finitely generated if and only if $G$ is finitely generated.  
\end{thm}
\begin{proof} "Only if part" is clear. "If part" follows from a standard method finding a set of generators of $H$ from a given set of generators of $G$ and complete representatives of the left coset $G/H$, called Reidemeister's method. See e.g. \cite[Page 181, Corollary 1]{Su82} for a self-contained proof.
\end{proof}

The following lemma is implicitly used in several papers. Our argument here is much inspired by a paper of Professor J\'anos Koll\'ar \cite[Proof of Theorem 6]{Ko09}:
\begin{lemma}\label{lem21}
Let $K$ be an algebraically closed field and let $V$ be a projective variety defined over $K$. Assume that ${\rm Aut}\, (V/K)$ is discrete. Then, ${\rm Aut}\, (V/K) = {\rm Aut}\, (V_{L}/L)$, as groups, for any field extension $K \subset L$, under the natural inclusion ${\rm Aut}\, (V/K) \subset {\rm Aut}\, (V_{L}/L)$.  
\end{lemma}
\begin{proof} Let $\varphi \in {\rm Aut}\, (V_{L}/L) \setminus {\rm Aut}\, (V/K)$. Since $K$ is algebraicaly closed and $\varphi \not\in {\rm Aut}\, (V/K)$, the residue field of the point 
$$[\varphi] : {\rm Spec}\, L \to {\rm Aut}\, (X_L/L) \subset {\rm Hilb}\, (V_L \times V_L)$$ corresponding to the graph of $\varphi$ is transcendental over $K$. However, then, the specialization gives a positive dimensional subset of ${\rm Aut}\, (V/K) \subset {\rm Hilb}\, (V \times V)$, a contradiction to our assumption that ${\rm Aut}\, (V/K)$ is discrete. This implies the result.
\end{proof}

\begin{remark}\label{rem21}
Needless to say, ${\rm Aut}\, (V_{L}/L)$ is much bigger than ${\rm Aut}\, (V/K)$ in general. For instance, for an elliptic curve $E$ defined over $\overline{\Q}$, the group ${\rm Aut}\, (E_{\C}/\C)$ is uncountable, while ${\rm Aut}\, (E/\overline{\Q})$ is countable. 
\end{remark}

We denote by $\kappa (X)$ the Kodaira dimension of a smooth projective variety $X$. 

\begin{theorem}\label{thm22}
Let $K$ be an algebraically closed field.
\begin{enumerate}
\item Let $X$ be a variety defined over $K$. We assume that $X$ is a smooth projective surface such that either $\kappa (X) \ge 1$ or the image of the albanese morphism $X \to {\rm Alb}(X)$ is a curve, or $X$ is an abelian variety. Then the group ${\Aut}(X/K)/{\rm Aut}^0(X/K)$ is finitely generated. In particular, the automorphism group ${\rm Aut}_{{\rm group}}(X/K)$ of an abelian variety $X$ as a group variety is finitely generated.
 
\item Let $X$ be a smooth minimal projective surface defined over $K$. Assume that $\kappa (X) = 0$ and $K$ is of odd characteristic. Then the group ${\Aut}(X/K)/{\rm Aut}^0(X/K)$ is finitely generated.

\item Let $Y$ be a smooth projective surface defined over $k_0$, an algebraic closure of the prime field $\F_p$ of odd characteristic. Then, the group ${\Aut}(Y/k_0)/{\rm Aut}^0(Y/k_0)$ is finitely generated unless $Y$ is a rational surface.
\end{enumerate}
\end{theorem}

\begin{proof} The assertion (1) should be known for the experts. We give a proof for the convenience of the readers. Let 
$$\tau : {\rm Aut}(X/K)/{\rm Aut}^0(X/K) \to {O}({\rm NS}(X))/({\rm torsion})\,\, ;\,\, f \mapsto f^*|_{{\rm NS}(X)/({\rm torsion})}$$
be the natural contravariant group homomorphism. 
Then ${\rm Ker}\,(\tau)$ is a finite group by \cite[Theorem 2.10]{Br19}. Thus, it suffices to show that the group ${\rm Im}\, (\tau)$ is finitely generated. 

First consider the case where $X$ is a surface. 

Assume that $\kappa (X) \ge 1$. Let $mK_X = P + N$ be the Zariski decomposition of $mK_X$ where $m$ is a sufficiently divisible positive integer. Then $P \in {\rm NS}(X)/({\rm torsion})$. We have $(P^2) > 0$ when $\kappa (X) =2$ and $P \not= 0$ and $(P^2) = 0$ when $\kappa (X) =1$. Since the class $P$ is preserved by ${\rm Im}\,(\tau)$, it follows from the Hodge index theorem that ${\rm Im}\,(\tau)$ is finite when $\kappa (X) =2$ and ${\rm Im}\,(\tau)$ is a finitely generated abelian group up to finite kernel and cokernel when $\kappa (X) =1$ (See eg. \cite[Theorem 2.1]{Og07}). Hence the group ${\rm Im}\,(\tau)$ is finitely generated. 

Assume that the image of the albanese morphism $X \to {\rm Alb}(X)$ is a curve. Then the class of general fiber $F \in {\rm NS}(S)/({\rm torsion})$ satisfies $F \not= 0$ and $(F^2) = 0$ and is preserved by ${\rm Im}\,(\tau)$. Hence for the same reason as in $\kappa (X) =1$, the group ${\rm Im}\, (\tau)$ is a finitely generated abelian group up to finite kernel and cokernel, in particular, finitely generated. 

Next, consider the case where $X$ is an abelian variety. Let $O$ be the origin of the group $X$. Then ${\Aut}(X/K)/{\rm Aut}^0(X/K)$ is isomorphic to ${\Aut}(X/K, O) = {\rm Aut}_{{\rm group}}(X/K)$, which is an arithmetic subgroup of the real linear algebraic group $({\rm End}^0\,(X) \otimes \R)^{\times}$ defined over $\Q$ (See eg. \cite[Corollary 3.6]{PS12}). Hence the group ${\Aut}(X/K)/{\rm Aut}^0(X/K)$ is finitely generated by \cite[Theorem 6.2]{BH62}. 

This completes the proof of the assertion (1).

Let us show the assertion (2). By (1) and by the classification of smooth projective surfaces (\cite{BM77}), we may assume that $X$ is either a K3 surface or an Enriques surface. Since $K$ is of odd characteristic, the result follows from Theorem \ref{thm2} (the main result of \cite{LM18}) for K3 surfaces and \cite[Theorem 1.3]{Wa19} for Enriques surfaces. This completes the proof of the assertion (2). 

We show the assertion (3). Note that $Y$ is not a rational surface by our assumption. Then, by (1) and (2) and by the classification of smooth projective surfaces (\cite{BM77}), we may assume that $Y$ is not minimal and $Y$ is birational to either a K3 surface, an Enriques surface, or an abelian surface.

Let $X$ be the minimal model of $Y$. Then the surface $X$ is unique up to isomorphism and we have a birational morphism 
$$\pi = \pi_n \circ \pi_{n-1} \circ \ldots \circ \pi_0 : Y := X_{n+1} \to X_{n} \to \ldots \to X_1 \to X_0 := X\,\, ,$$
where $\pi_{i} : X_{i+1} \to X_i$ is the blow-up at some point $P_i \in X_i(k_0)$. For the same reason as in Remark \ref{rem1} (2), ${\rm Aut}\, (Y/k_0) \subset {\rm Aut}\, (X/k_0)$ via $\pi$. Let $E_{\pi} \subset Y$ be the exceptional set of $\pi$. Since $Y$ is not minimal, $\pi(E_{\pi})$ is a non-empty finite set of points, hence $P_0 \in \pi(E_{\pi})$, and the group ${\rm Aut}\, (Y/k_0)$ preserves $\pi(E_{\pi})$ via $\pi$. Then, 
$$H := {\rm Aut}\, (Y/k_0, \pi^{-1}(P_0))$$
is a finite index subgroup of ${\rm Aut}\, (Y/k_0)$ such that $H \subset {\rm Aut}\, (X/k_0, P_0)$ via $\pi$. 

We are going to show that $H$ is finitely generated. First we observe the following:

\begin{lemma}\label{lem22} ${\rm Aut}\, (X/k_0, P_0)$ is a finitely generated group.
\end{lemma}
\begin{proof} The result follows from Theorem \ref{thm22} (1) when $X$ is an abelian surface. Indeed, we may take $P_0$ as the origin of $X$. 

Assume that $X$ is a K3 surface or an Enriques surface. Then ${\rm Aut}\, (X/k_0)$ is finitely generated by Theorem \ref{thm22} (2). We set 
$${\rm Aut}\, (X/k_0) = \langle h_1, \ldots , h_r \rangle\,\, .$$
Then there is a positive integer $q$, which is a power of $p$, such that $h_j$ ($1 \le j \le r$) are all defined over $\F_{q}$ and also $P_0 \in X(\F_{q})$. By definition, any $h \in {\rm Aut}(X/k_0)$ is then defined over $\F_q$. Let $T = X(\F_q)$. Then $T$ is a finite set. Since $T$ is preserved by each $h_j$, it follows that $T$ is preserved by ${\rm Aut}\, (X/k_0)$. Hence we have a group homomorphism
$$\sigma : {\rm Aut}\, (X/k_0) \to {\rm Aut}_{{\rm set}}(T)$$
and 
$${\rm Ker}\, (\sigma) \subset {\rm Aut}\, (X/k_0, P_0) \subset {\rm Aut}\, (X/k_0).$$
Then ${\rm Ker}\, (\sigma)$ is a finite index subgroup of ${\rm Aut}\, (X/k_0)$ by $|T| < \infty$. Hence ${\rm Aut}\, (X/k_0, P_0)$ is a finite index subgroup of ${\rm Aut}\, (X/k_0)$ as well. Since ${\rm Aut}\, (X/k_0)$ is finitely generated, so is ${\rm Aut}\, (X/k_0, P_0)$ by Theorem \ref{thm21}. This completes the proof of the assertion (3).
\end{proof}
By Lemma \ref{lem22}, we may set 
$${\rm Aut}\, (X/k_0, P_0) = \langle g_1, \ldots , g_m \rangle\,\, .$$
Then there is a positive integer $q$, which is a power of $p$, such that $g_j$ are all defined over $\F_{q}$ and also $P_i \in X_i(\F_{q})$ for all integers $0 \le i \le n+1$ and $1 \le j \le m$. By definition, any $g \in {\rm Aut}(X/k_0, P_0)$ is then defined over $\F_q$ and the blow-up $\pi_i$ are also defined 
over $\F_{q}$.  

Let $S = X(\F_q)$. Then $S$ is a finite set. We consider the blow-up $p_0 : Y_1 \to Y_0 := X$ at $S$ and the exceptional divisor $E_S$ of $p_0$. Here $E_S$ is a disjoint union of $|S|$ $\P^1$s. Then $Y_1$ is defined over $\F_q$ and $S_1 := E_S(\F_q)$ is a finite set. We then consider the blow-up $p_1 : Y_2 \to Y_1$ at $S_1$. Then $Y_2$ is defined over $\F_q$. We repeat this process $(n+1)$-times, where $n$ is the same positive integer $n$ as in $\pi : Y \to X$ above, and get the birational morphism
$$\varphi := p_{n} \circ p_{n-1} \circ \ldots \circ p_0 : Z := Y_{n+1} \to Y_{n} \to \ldots \to Y_1 \to Y_0 = X\,\, .$$
By the choice of $\F_q$ and by the construction of $Z$, each element of ${\rm Aut}\, (X/k_0, P_0)$ lifts to an element of ${\rm Aut}\, (Z/k_0, \varphi^{-1}(P_0))$ under $\varphi$. Thus the inclusion 
$${\rm Aut}\, (Z/k_0, \varphi^{-1}(P_0)) \subset {\rm Aut}\, (X/k_0, P_0)$$
induced by ${\rm Aut}\, (Z/k_0) \subset {\rm Aut}\, (X/k_0)$ via $\varphi$ is actually an equality, that is, 
$${\rm Aut}\, (Z/k_0, \varphi^{-1}(P_0))) = {\rm Aut}\, (X/k_0, P_0)$$ 
via $\varphi$. Let $\{E_j\}_{j=1}^{N}$ be the set of the irreducible components of the exceptional divisor of $\varphi$. By construction, the set $\{E_j\}_{j=1}^{N}$ is preserved by ${\rm Aut}\, (X/k_0, P_0)$ under the identification made above. Thus, we have a group homomorphism 
$$\rho : {\rm Aut}\, (X/k_0, P_0) \to {\rm Aut}_{{\rm set}}(\{E_i\}_{i=1}^{N}) \simeq S_N\,\, .$$
Here $S_N$ is the symmetric group of $N$ letters. Let $G = {\rm Ker}\, (\rho)$. Then 
$$[{\rm Aut}\, (X/k_0) : G]  = |{\rm Im}\, \rho | \le |S_N| = N! < \infty\,\, .$$ 
On the other hand, again by our choice of $\F_q$ and the construction of $Z$, we have the factorization $\tau : Z \to Y$ of $\varphi : Z \to X$ by $\pi : Y \to X$:
$$\varphi = \pi \circ \tau :  Z \to Y \to X\,\, .$$
Then $\tau$ is the smooth blow-down of some irreducible curves in $\{E_j\}_{j=1}^{N}$. 
By the definition, $G$ preserves each element $E_j$ of $\{E_j\}_{j=1}^{N}$. Thus any element of $G$ descends to $H = {\rm Aut}\, (Y, \pi^{-1}(P_0))$ via $\tau$. Hence we have the following group inclusions
$$G \subset H \subset {\rm Aut}\, (X/k_0, P_0)$$
via $\tau$ and $\pi$. The resulting inclusion $G \subset {\rm Aut}\, (X/k_0, P_0)$ is then the same as the one via $\varphi = \pi \circ \tau$. Thus 
$$[{\rm Aut}\, (X/k_0, P_0) : H] \le [{\rm Aut}\, (X/k_0, P_0) : G] < \infty\,\, .$$
Recall that ${\rm Aut}\, (X/k_0, P_0)$ is finitely generated (Lemma \ref{lem22}). Hence by Theorem \ref{thm21}, $H$ is finitely generated. Since $H$ is a finite index subgroup of ${\rm Aut}\, (Y/k_0)$, the group ${\rm Aut}\, (Y/k_0)$ is also finitely generated as well by Theorem \ref{thm21}. This completes the proof of the assertion (3). 
\end{proof}

We are ready to prove Theorem \ref{thm1} (1). We use the same notation as in Theorem \ref{thm1} (1). Since $Y$ is birational to a K3 surface, ${\rm Aut}\,(Y/k_0)$ is discrete (Remark \ref{rem1} (2)). Thus, by Theorem \ref{thm22} (3), ${\rm Aut}\, (Y/k_0)$ is a finitely generated group. Hence ${\rm Aut}\, (Y_L/L)$ is finitely generated as well by Lemma \ref{lem21}. This completes the proof of Theorem \ref{thm1} (1).

\section{Proof of Theorem \ref{thm1} (2)}\label{sect3}

In this section, we prove Theorem \ref{thm1}(2) by constructing $Y$ explicitly from an explicitly given Kummer K3 surface $X$ below. Our main result of this section is Theorem \ref{thm33}. As mentioned in the introduction, our construction is very close to the one in \cite{DO19}.  
 
Let $k$ be an algebraically closed field as in Introduction. Recall that$$t \in \F_p(t) \subset k$$
and $t$ is transcendental over $\F_p$. 

We finally reduce our proof of non-finite generation to the following lemma. 

\begin{lemma}\label{lem39} The subgroup $G_t := \langle t^n | n \in \Z \rangle$ of the additive group $k = (k, +)$ is not finitely generated. 
\end{lemma}

\begin{proof} If otherwise, $G_t$ would be a finitely generated abelian group with $\F_p$-vector space structure induced by the one on $k$. So $G$ has to be a finite dimensional $\F_p$-vector space, say of dimension $d$. Then the following $d+1$ elements 
$$1\,\, ,\,\, t\,\, ,\,\, t^2\,\, ,\,\, \ldots ,\,\, t^{d}$$
of $G_t$ has to be linearly dependent over $\F_p$. Thus, there is 
$$(0,0, \ldots , 0) \not= (a_0, a_1,\ldots , a_d) \in \F_p^{\oplus d}$$
such that 
$$a_0 + a_1t + \ldots + a_dt^d = 0$$
in $G_t \subset k$. However, this contradicts to the fact that $t$ is transcendental over $\F_p$. 
\end{proof}

Let $E$ be the elliptic curve defined over $k$ by the Weierstrass equation
$$y^2 = x(x-1)(x-t)\,\, .$$
Note that $E/\langle -1_E \rangle = \P^1$, the associated quotient map $E \to \P^1$ is given by $(x,y)\mapsto x$ and the points $0$, $1$, $t$ and $\infty$ of $\P^1(k)$ are exactly the branch points of this quotient map. 

Let $F$ be any elliptic curve defined over $k$ such that $F$ is not isogenous to $E$. For instance, we may take a supersingular elliptic curve defined over $k$ as $F$. Note that there certainly exists a supersingular elliptic curve $F$ over $k$ and $E$ is not a supersingular (see eg. subsection "Elliptic curves in Characteristic $p >0$" in \cite[Section 22]{Mu74}). In particular, $E$ and $F$ are not isogenous over $k$ (see eg. subsection "The $p$-rank" in \cite[Section 15]{Mu74}).

Throughout this section, we denote by 
$$X := {\rm Km} (E \times F)$$ 
the Kummer K3 surface accociated to the product abelian surface $E \times F$, that is, the minimal resolution of the quotient surface $E \times F/\langle (-1_E, -1_F) \rangle$. We write $H^0(X, \Omega_X^2) = k\omega_X$. Then $\omega_X$ is a nowhere vanishing regular global $2$-form on $X$ and it is induced by a nowhere vanishing regular global $2$-form on $E \times F$. 

Since $E$ and $F$ are not isogenous, the Picard number $\rho(E \times F)$ of $E \times F$ is $2$ and therefore the Picard number $\rho(X)$ of $X$ is $18$ by \cite[Proposition 1 and Appendix]{Sh75}. In particular, our K3 surface $X$ is not supersingular. 

Let $\{a_i\}_{i=1}^{4}$ and $\{b_i\}_{i=1}^{4}$ be the $2$-torsion subgroups of $F$ and $E$ respectively. Then $X$ contains 24 "visible" smooth rational curves as in Figure \ref{fig1}. Here smooth rational curves $E_i$, $F_i$ ($1 \le i \le 4$) are arising from the elliptic curves $E \times \{a_i\}$, $\{b_i\} \times F$ on $E \times F$. Smooth rational curves $C_{ij}$ ($1\le i,j \le 4$) are the exceptional curves over the $A_1$-singular points of the quotient surface $E \times F/\langle -1_{E \times F}\rangle$. Throughout this section, we will freely use the names of curves in Figure \ref{fig1}. 

\begin{figure}
\unitlength 0.1in
\begin{picture}(25.000000,24.000000)(-1.000000,-23.500000)
\put(4.500000, -22.000000){\makebox(0,0)[rb]{$F_1$}}%
\put(9.500000, -22.000000){\makebox(0,0)[rb]{$F_2$}}%
\put(14.500000, -22.000000){\makebox(0,0)[rb]{$F_3$}}%
\put(19.500000, -22.000000){\makebox(0,0)[rb]{$F_4$}}%
\put(0.250000, -18.500000){\makebox(0,0)[lb]{$E_1$}}%
\put(0.250000, -13.500000){\makebox(0,0)[lb]{$E_2$}}%
\put(0.250000, -8.500000){\makebox(0,0)[lb]{$E_3$}}%
\put(0.250000, -3.500000){\makebox(0,0)[lb]{$E_4$}}%
\put(6.000000, -16.000000){\makebox(0,0)[lt]{$C_{11}$}}%
\put(6.000000, -11.000000){\makebox(0,0)[lt]{$C_{12}$}}%
\put(6.000000, -6.000000){\makebox(0,0)[lt]{$C_{13}$}}%
\put(6.000000, -1.000000){\makebox(0,0)[lt]{$C_{14}$}}%
\put(11.000000, -16.000000){\makebox(0,0)[lt]{$C_{21}$}}%
\put(11.000000, -11.000000){\makebox(0,0)[lt]{$C_{22}$}}%
\put(11.000000, -6.000000){\makebox(0,0)[lt]{$C_{23}$}}%
\put(11.000000, -1.000000){\makebox(0,0)[lt]{$C_{24}$}}%
\put(16.000000, -16.000000){\makebox(0,0)[lt]{$C_{31}$}}%
\put(16.000000, -11.000000){\makebox(0,0)[lt]{$C_{32}$}}%
\put(16.000000, -6.000000){\makebox(0,0)[lt]{$C_{33}$}}%
\put(16.000000, -1.000000){\makebox(0,0)[lt]{$C_{34}$}}%
\put(21.000000, -16.000000){\makebox(0,0)[lt]{$C_{41}$}}%
\put(21.000000, -11.000000){\makebox(0,0)[lt]{$C_{42}$}}%
\put(21.000000, -6.000000){\makebox(0,0)[lt]{$C_{43}$}}%
\put(21.000000, -1.000000){\makebox(0,0)[lt]{$C_{44}$}}%
\special{pa 500 2200}%
\special{pa 500 0}%
\special{fp}%
\special{pa 1000 2200}%
\special{pa 1000 0}%
\special{fp}%
\special{pa 1500 2200}%
\special{pa 1500 0}%
\special{fp}%
\special{pa 2000 2200}%
\special{pa 2000 0}%
\special{fp}%
\special{pa 0 1900}%
\special{pa 450 1900}%
\special{fp}%
\special{pa 550 1900}%
\special{pa 950 1900}%
\special{fp}%
\special{pa 1050 1900}%
\special{pa 1450 1900}%
\special{fp}%
\special{pa 1550 1900}%
\special{pa 1950 1900}%
\special{fp}%
\special{pa 0 1400}%
\special{pa 450 1400}%
\special{fp}%
\special{pa 550 1400}%
\special{pa 950 1400}%
\special{fp}%
\special{pa 1050 1400}%
\special{pa 1450 1400}%
\special{fp}%
\special{pa 1550 1400}%
\special{pa 1950 1400}%
\special{fp}%
\special{pa 0 900}%
\special{pa 450 900}%
\special{fp}%
\special{pa 550 900}%
\special{pa 950 900}%
\special{fp}%
\special{pa 1050 900}%
\special{pa 1450 900}%
\special{fp}%
\special{pa 1550 900}%
\special{pa 1950 900}%
\special{fp}%
\special{pa 0 400}%
\special{pa 450 400}%
\special{fp}%
\special{pa 550 400}%
\special{pa 950 400}%
\special{fp}%
\special{pa 1050 400}%
\special{pa 1450 400}%
\special{fp}%
\special{pa 1550 400}%
\special{pa 1950 400}%
\special{fp}%
\special{pa 200 2000}%
\special{pa 600 1600}%
\special{fp}%
\special{pa 200 1500}%
\special{pa 600 1100}%
\special{fp}%
\special{pa 200 1000}%
\special{pa 600 600}%
\special{fp}%
\special{pa 200 500}%
\special{pa 600 100}%
\special{fp}%
\special{pa 700 2000}%
\special{pa 1100 1600}%
\special{fp}%
\special{pa 700 1500}%
\special{pa 1100 1100}%
\special{fp}%
\special{pa 700 1000}%
\special{pa 1100 600}%
\special{fp}%
\special{pa 700 500}%
\special{pa 1100 100}%
\special{fp}%
\special{pa 1200 2000}%
\special{pa 1600 1600}%
\special{fp}%
\special{pa 1200 1500}%
\special{pa 1600 1100}%
\special{fp}%
\special{pa 1200 1000}%
\special{pa 1600 600}%
\special{fp}%
\special{pa 1200 500}%
\special{pa 1600 100}%
\special{fp}%
\special{pa 1700 2000}%
\special{pa 2100 1600}%
\special{fp}%
\special{pa 1700 1500}%
\special{pa 2100 1100}%
\special{fp}%
\special{pa 1700 1000}%
\special{pa 2100 600}%
\special{fp}%
\special{pa 1700 500}%
\special{pa 2100 100}%
\special{fp}%
\end{picture}%
 \caption{Curves $E_i$, $F_j$ and $C_{ij}$}
 \label{fig1}
\end{figure}
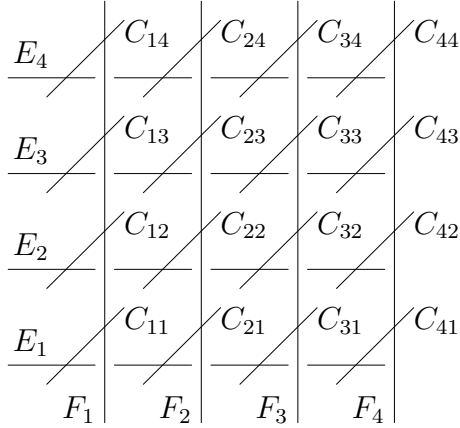

\begin{definition}\label{def31}
As in \cite{DO19}, we set 
$$C := E_1 = E/\langle -1_E \rangle \subset X\,\, .$$
We may and do use $x$ in the Weierstrass equation of $E$ as an affine coordinate on $C$ and also assume that under the affine coordinate $x$,  
$$C \cap C_{11} = \{\infty\}\,\, , \,\, C \cap C_{21} = \{0\}\,\, ,\,\, C \cap C_{31} = \{1\} \,\, , \,\, C \cap C_{41} = \{t\}\,\, .$$
We define the point $P \in C(k) \subset X(k)$ by
$$P:= \infty\,\, ,$$
that is, the intersection point of $C$ and $C_{11}$. 
\end{definition}

Let 
$$\theta = [(1_E, -1_F)] = [(-1_E, 1_F)] \in {\rm Aut}\, (X/k)$$ 
be the automorphism of $X$ induced by the automorphism $(1_E, -1_F) \in {\rm Aut}\, ((E \times F)/k)$ of 
$E \times F$. Then $\theta$ is of order $2$. Set 
$$B := \cup_{i=1}^{4} E_i \cup \cup_{j=1}^{4} F_j\,\, .$$
The following theorem was proved in \cite[Lemmas (1.3), (1.4)]{Og89} over $\C$. However, the proof there is based on the global Torelli theorem for complex K3 surfaces and Hodge theory. So, one can not apply the argument there for our $X$ over $k$.

\begin{thm}\label{thm31} 
The following properties hold also over $k$.
\begin{enumerate}
\item The Picard group ${\rm Pic}\, (X)$ is torsion free.
\item $\theta^* = \id$ on ${\rm Pic}\, (X)$ and $\theta^* \omega_X = -\omega_X$.\item ${\rm Aut}\, ((E \times F)/k) = {\rm Aut}\, (E/k) \times {\rm Aut}\, (F/k)$ and $f \circ \theta = \theta \circ f$ for all $f \in {\rm Aut}\,(X/k)$. 
\item Let $X^{\theta}$ be the fixed locus of $\theta$. Then $X^{\theta} = B$. 
\item ${\rm Aut}\,(X/k) = {\rm Aut}\, (X/k, B)$. 
\end{enumerate}
\end{thm}
\begin{proof} Assume that $L \in {\rm Pic}\, (X)$ satisfies $nL =0$ in ${\rm Pic}\, (X)$ for some positive integer $n$. Then $(L, L)_X = 0$ and therefore $\chi(X, L) = 2$ by the Riemann-Roch formula. Combining this with the Serre duality, we deduce that either $L$ or $-L$ is represented by an effective divisor. This implies $L = 0$, as $(\pm L, H)_X = 0$ for a very ample divisor $H$ on $X$ by $nL = 0$ ($n \not= 0$). This proves (1).  

By $\rho(X) = 18$, we see that ${\rm Pic}\, (X) \otimes \Q$ is generated by the $24$ rational curves in Figure \ref{fig1}. It is clear that $\theta$ preserves each of these $24$ curves. It follows that  $\theta^* = \id$ on ${\rm Pic}\, (X) \otimes \Q$ and therefore $\theta^* = \id$ also on ${\rm Pic}\, (X)$ by (1). By the shape of $\theta$, clearly $\theta^*\omega_X = -\omega_X$. This shows (2). 

The first assertion of (3) is an immediadte consequence of our assumption that $E$ and $F$ are not isogenous. Let 
$$g := \theta \circ f \circ \theta^{-1} \circ f^{-1} \in {\rm Aut}\, (X/k)\,\, .$$ We are going to show that $g = \id_X$. We have $g^{*} = id$ on ${\rm Pic}\, (X)$ by (1) and $g^*\omega_X = \omega_X$ by the definition of $g$. In particular, $g(R) = R$ for all smooth rational curves $R \subset X$. This is because $g^*(R) = R$ in ${\rm Pic}\, (X)$ and $|R| = \{R\}$ by $(R^2) = -2$. Here we recall that $(R^2) = -2$ for any smooth rational curve $R$ on $X$ by the adjunction formula. Thus 
$$g(\sum_{i,j} C_{ij}) = \sum_{i,j} C_{ij}\,\, .$$ 
Let $V$ be the blow-up of $E \times F$ at the sixteen $2$-torsion points $P_{ij}$ and $D_{ij} \subset V$ the exceptional curve over the $2$-torsion point $P_{ij}$. The induced morphism $\pi : V \to X$ is a finite cover of degree $2$ branched along $\sum_{i,j} C_{ij}$. Since ${\rm Pic}\, (X)$ is torsion free and since any degree $2$ map is separable over $k$ of odd characteristic, the converse is also true. That is, if $\pi' : V' \to X$ is a finite double cover branched along $\sum_{i,j} C_{ij}$, then $\pi : V \to X$ and $\pi' : V' \to X$ are isomorphic over $X$ (See eg. \cite[Theorem 2.6]{Fu83}). Applying this for $g \circ \pi : V \to X$ and $\pi : V \to X$, we deduce that $g$ lifts to an automorphism $g_V$ of $V$ such that $g_V(D_{ij}) = D_{ij}$ for each $(i, j)$. Then $g_V$ descends to the automorphism $g_{E \times F}$ of $E \times F$ such that 
$$g_{E \times F}(P_{ij}) = P_{ij}$$ 
for each $2$-torsion point $P_{ij}$ of $E \times F$ and 
$$g_{E \times F}^*\omega_{E \times F} = \omega_{E \times F}\,\, .$$ 
Write $g_{E \times F} = (g_E, g_F)$ ($g_E \in {\rm Aut}\, (E/k)$, $g_F \in {\rm Aut}\, (F/k)$) by using the first assertion of (3). Then $g_E = \pm \id_E$ and $g_F = \pm \id_F$ by $g_{E \times F}(P_{ij}) = P_{ij}$. Combining this with $g_{E \times F}^*\omega_{E \times F} = \omega_{E \times F}$, we obtain that $g_{E \times F} = \pm \id_{E \times F}$. Hence $g = \id_X$, i.e., $\theta \circ f = f \circ \theta$ on $X$. This proves (3). 

The assertion (4) is immediate from the shape of $\theta$. The assertion (5) follows from (3) and (4). This completes the proof.
\end{proof}

\begin{proposition}\label{prop32} 
${\rm Aut}\, (X/k, P) \subset {\rm Aut}\, (X/k, C)$. That is, $f(C) = C$ holds for every $f \in {\rm Aut}\, (X/k, P)$. In particular, 
$${\rm Aut}\, (X/k, P) = {\rm Aut}\, (X/k, C, P)\,\, .$$
\end{proposition}
\begin{proof} By Theorem \ref{thm31}, we have ${\rm Aut}\,(X) = {\rm Aut}\,(X, B)$. This implies the result, because $C$ is the unique irreducible component of $B$ such that $P \in C(k)$.  
\end{proof}

\begin{lemma}\label{lem33} Let $R \subset X$ be a smooth rational curve such that $R \not\subset B$. Then
\begin{enumerate}
\item $\theta(R) = R$ and $\theta|_R \in {\rm Aut}\, (R/k)$ is of order $2$. Moreover, $\theta|_R$ has exactly two fixed closed points and $d(\theta|_R)_Q = -1$ at each fixed closed point $Q \in R(k)$ of $\theta|_R$. 
\item Assume furthermore that $P \in R(k)$. Then, for each $f \in {\rm Aut}\, (X/k, P)$, either $f(R) = R$ or $f(R)$ and $R$ are tangent at $P$. 
\end{enumerate}
\end{lemma}
\begin{proof} Note that $-1 \not= 1$ in the field $k$ of odd characteristic. So, once Theorem \ref{thm31} is established, then exactly the same proof as \cite[Lemma 3.5]{DO19} works also over $k$.
\end{proof}

Recall that ${\rm Aut}\, (X/k, P) = {\rm Aut}\, (X/k, C, P)$ (Proposition \ref{prop32}). We define two differential representations of ${\rm Aut}\,(X/k, P)$, $d_{X, P}$ on the tangent space $T_{X, P} \simeq k^2$ and $d_{X, C, P}$ on the tangent space $T_{C, P} \simeq k$, and two subgroups $G(X, P)$ and $G(X, C, P)$ of ${\rm Aut}\, (X/k, P)$ by
$$d_{X, P} : {\rm Aut}\, (X/k, P) \to {\rm GL}(T_{X, P})\,\, ;\,\, f \mapsto df_P\,\, ,$$
$$d_{X, C, P} : {\rm Aut}\, (X/k, P) \to {\rm GL}(T_{C, P})\,\, ;\,\, f \mapsto d(f|_C)_P\,\, ,$$
$$G(X, P) := {\rm Ker}\, (d_{X, P} : {\rm Aut}\, (X/k, P) \to {\rm GL}(T_{X, P})\,\, ;\,\, f \mapsto df_P)\,\, ,$$
$$G(X, C, P) := {\rm Ker}\, (d_{X, C, P} : {\rm Aut}\, (X/k, P) \to {\rm GL}(T_{C, P})\,\, ;\,\, f \mapsto d(f|_C)_P)\,\, .$$
Clearly $G(X, P) \subset G(X, C, P)$ as groups.

Let $0 \not= v_1 \in T_{C, P} \subset T_{X, P}$ and $0 \not= v_2 \in T_{C_{11}, P} \subset T_{X, P}$. 
Then $\langle v_1, v_2 \rangle$ forms a basis of the $k$-vector space $T_{X, P}$. 
\begin{proposition}\label{prop34} ${\rm Im}\, (d_{X, P})$ is simultaneously diagonalizable with respect to the basis $\langle v_1, v_2 \rangle$ of $T_{X, P}$. 
\end{proposition}
\begin{proof} This is because $df_P$ ($f \in {\rm Aut}\, (X/k, P)$) preserves $T_{C, P}$ and also preserves $T_{C_{11}, P}$ by Lemma \ref{lem33}(2).  
\end{proof}

Let $K$ be any algebraically closed field of odd characteristic and let $S$ be any K3 surface defined over $K$. Then we have $H^0(S, \Omega_S^2) = K \omega_S \simeq K$ and for each $f \in {\rm Aut}\, (S/K)$, there is a unique $\alpha(f) \in K^{\times}$ such that $f^{*}\omega_S = \alpha(f)\omega_S$. The group 
homomorphism 
$$\alpha : {\rm Aut}\, (S/K) \to {\rm GL}\, (K\omega_S) = K^{\times}\,\, ;\,\, f \mapsto \alpha(f)$$
is called the canonical representation of $S$ or of ${\rm Aut}\,(S/K)$.

\begin{thm}\label{thm32}
The image $\alpha({\rm Aut}\, (S/K))$ of the canonical representation is a finite group, hence a finite cyclic group, for any non-supersingular K3 surface $S$ defined over any algebraically closed field $K$ of odd characteristic. 
\end{thm}
\begin{proof} This is proved by Jang \cite[Proposition 3.5]{Ja13} as an important application of Theorem \ref{thm2} (2) due to Lieblich and Maulik. Here we recall the proof for the convenience of the readers. Let $\pi : X \to {\rm Spec}\, R$ be the lifting of $S$ in Theorem \ref{thm2} (2). Let $\omega_{X/R}$ be the relative regular $2$-form of $\pi$. Consider the canonical representation $\alpha_{S} = \alpha$ of $S$:
$$\alpha_{S} : {\rm Aut}\, (S/K) \to {\rm GL}\, (K \omega_{X/R}|_{S}) = K^{\times}\,\, .$$
Let $G$ be the image of the specialization map $${\rm Aut}\,(\tilde{S}/\tilde{K}) \to {\rm Aut}\,(S/K)\,\, $$
in Theorem \ref{thm2} (2).  Let ${\bf m}$ be the maximal ideal of $R$. Then the homomorphism
$$\alpha_{S}|_{G} : G \to {\rm GL}\, (K \omega_{X/R}|_{S}) = K^{\times}$$ is the mod ${\bf m}$-reduction of the canonical representation of $\tilde{S}$:
$$\alpha_{\tilde{S}} : {\rm Aut}\,(\tilde{S}/\tilde{K}) \to {\rm GL}\, (\tilde{K} \omega_{X/R}|_{\tilde{S}}) = \tilde{K}^{\times}\,\, .$$
Since $\tilde{K}$ is of characteristic $0$, the group ${\rm Im}\, \alpha_{\tilde{S}}$ is a finite cyclic group by \cite[Theorem 14.10]{Ue75}. Therefore ${\rm Im}\, \alpha_{S}|_{G} = \alpha_S(G)$ is also a finite cyclic group. Since $[{\rm Aut}\, (S/K) : G] < \infty$ by Theorem \ref{thm21} (2), it follows that $\alpha_S({\rm Aut}\, (S/K))$ is a finite subgroup of $K^{\times}$. Hence it is a finite cyclic group as claimed.
\end{proof}

Recall that $G(X, P)$ is a subgroup of $G(X, C, P)$.

\begin{proposition}\label{prop31}
\begin{enumerate}
\item $[G(X, C, P) : G(X, P)] < \infty$.
\item $G(X, P)$ is not finitely generated. 
\end{enumerate}
\end{proposition}

\begin{proof}

First, we prove the assertion (1). 
Let $f \in {\rm Aut}\, (X/k, P)$. Then by Proposition \ref{prop34}, we have 
$$df_P(v_1) = \alpha_1(f)v_1\,\, ,\,\, df_P(v_2) = \alpha_2(f)v_2$$
for some $\alpha_1(f), \alpha_2(f) \in k^{\times}$. Then for the canonical representation $\alpha$ of ${\rm Aut}\, (X/k)$, we have
$$\alpha(f) = \alpha_1(f)\alpha_2(f)\,\, .$$
Then $\alpha(f) = \alpha_2(f)$ for $f \in G(X, C, P)$, as $\alpha_1(f) =1$ for $f \in G(X, C, P)$. Thus
$$G(X, P) = {\rm Ker}\, (\alpha|_{G(X, C, P)})$$
by Proposition \ref{prop34}, and therefore
$$[G(X,C, P) : G(X, P)] = |{\rm Im}\, (\alpha|_{G(X, C, P)})| \le |{\rm Im}\, (\alpha)| < \infty\,\, ,$$
by Theorem \ref{thm32}. This completes the proof of the assertion (1).  

Next we shall prove the assertion (2). Consider the group representation
$$\tau : G(X, C, P) \to {\rm Aut}(C, P)\,\, ;\,\, f \mapsto f|_C\,\, .$$
Let
$$\Gamma := \tau(G(X, C, P))\,\, ,\,\, f \in G(X, C, P)\,\, .$$ 
Then by the definition of $G(X, C, P)$, we have 
$f|_C(P) = P$ and $d(f|_C)_P = 1$ on $C = \P^1$. 
Under the affine coordinate $x$ of $C$, the automorphism $f|_C \in {\rm Aut}\,(C/k)$ is then of the form 
$$f(x) = x + a_f\quad \text{with} \quad a_f \in k\,\, .$$
Thus $\Gamma$ is isomorphic to a subgroup of the additive group $k = (k,+)$ and therefore $\Gamma$ is an abelian group with $\F_p$-linear space structure. 

Now, to conclude Proposition \ref{prop31} (2), it suffices to show that $\Gamma$ {\it do have} a non-finitely generated subgroup. Indeed, then, $\Gamma$ is not also finitely generated, as $\Gamma$ is an abelian group. Hence $G(X, C, P)$ is not finitely generated, as its image $\Gamma := \tau(G(X, C, P))$ is not finitely generated. Since $[G(X, C, P) : G(X, P)] < \infty$ by Proposition \ref{prop31},  $G(X, P)$ is not finitely generated as well, by Theorem \ref{thm21}.

In the rest, we will find a non-finitely generated subgroup of $\Gamma$ by constructing various (quasi-)elliptic fibrations with section on $X$. 
 
As in \cite{DO19}, consider the following two divisors $D_1$ and $D_2$ of Kodaira's type $I_8$ and $IV^*$ on $X$:
$$D_1 := C + C_{11} + F_1 + C_{12} + E_2 + C_{22} + F_2 + C_{21}\,\, ,$$
$$D_2 := C + 2C_{11} + E_2 + 2C_{12} + E_3 + 2C_{13} + 3F_1\,\, .$$
Observe also that 
$$(D_1.C_{31}) = (D_1.C_{41}) = 1\,\, ,\,\, (D_2.C_{21}) = (D_2.C_{31}) = 1\,\, .$$
Thus, by \cite[Prop. 3.8]{DO19}, which is also valid over any algebraically closed field $K$ (if one replaces the term "elliptic" there by "quasi-elliptic" when $K$ is of characteristic $2$, $3$), we obtain two (quasi-)elliptic fibrations 
$$\varphi_{D_1} : X \to \P^1$$
with $D_1$ as a singular fiber and two global sections $C_{31}$, $C_{41}$ meeting $C$, and 
$$\varphi_{D_2} : X \to \P^1$$
with $D_2$ as a singular fiber and two global sections $C_{21}$, $C_{31}$ meeting $C$. 

Choose $C_{31}$ as the zero section of $\varphi_{D_1}$ and $C_{21}$ as the zero section of $\varphi_{D_2}$. We now consider the Mordell-Weil groups ${\rm MW}\, (\varphi_{D_i})$ ($i=1$, $2$), that is, the group of the global sections of $\varphi_i$. Then ${\rm MW}\, (\varphi_i)$ is an abelian subgroup of ${\rm Aut}\,(X/k) = {\rm Bir}\,(X/k)$. 

Let $f_1$ and $f_2$ denote the automorphisms of $X$ given respectively by 
$C_{41} \in {\rm MW}(\varphi_{D_1})$ and $C_{31} \in {\rm MW}(\varphi_{D_2})$.  
As in the complex case \cite{Ko63} (see also \cite[Prop. 3.9]{DO19}), by a result of N\'eron (\cite{Ne64}), $f_1$ acts on 
$$C(k) \setminus ({\rm Sing}\, C)(k) = C(k) \setminus \{0, \infty\} = \G_{m}(k) = k^{\times}$$ 
by the multiplication by $t$ and $f_2$ acts on 
$$C(k) \setminus ({\rm Sing}\, C)(k) = C(k) \setminus \{\infty\} = \G_{a}(k) = k$$ 
by the addition by $1$, with respect to the affine coordinate $x$ of $C$ and the coordinate values $C \cap C_{11} = \{\infty\}$, $C \cap C_{21} = \{0\}$, $C \cap C_{31} = \{1\}$, $C \cap C_{41} = \{t\}$ in Definition \ref{def31}.
 
In particular, both $f_i$ ($i=1, 2$) preserve $C$ and the induced actions $f_i|_C \in {\rm Aut}\,(C/k)$ on $C$ are given, under the coordinate $x$, by 
$$f_1|_C(x) = tx \,\, ,\,\, f_2|_C(x) = x+1\,\, .$$
Thus 
$$(f_1|_C)^{n} \circ (f_2|_C) \circ (f_1|_C)^{-n}(x) =  x + t^n\,\, ,$$
i.e., the additive translation by $t^n$, and therefore 
$$f_1^{n} \circ f_2 \circ f_1^{-n} \in G(X, C, P)\,\, {\rm and}\,\, (f_1|_C)^{n} \circ (f_2|_C) \circ (f_1|_C)^{-n}\in \Gamma$$ 
for any integer $n$. 
Consider the following subgroup 
$$\Gamma_1 := \big\langle (f_1|_C)^{n} \circ (f_2|_C) \circ (f_1|_C)^{-n}\, |\, n \in \Z \big\rangle$$ 
of $\Gamma$. By the description above, $\Gamma_1$ is isomorphic to the group $G_t$ in Lemma \ref{lem39}. Thus, $\Gamma_1$ is not finitely generated by Lemma \ref{lem39}. This completes the proof of the second assertion (2).
\end{proof}

Let $\pi_1 : Y_1 \to X$ be the blow-up of $X$ at $P$ and $E_P \subset Y_1$ the exceptional curve. We choose $Q \in E_P(k) \setminus \{[v_1], [v_2]\}$. Here $v_1$ and $v_2$ are tangent directions of $C$ and $C_{11}$ at $P$. We then take the blow-up $\pi_2 : Y \to Y_1$ of $Y_1$ at $Q$. 

The following theorem completes the proof of Theorem \ref{thm1} (2). 

\begin{thm}\label{thm33}
${\rm Aut}(Y/k)$ is not finitely generated. 
\end{thm}

\begin{proof} As $f|_{T_{X, P}} = \id_{T_{X, P}}$ for $f \in G(X, P)$, we have 
$$G(X/k, P) \subset {\rm Aut}\,(Y/k) \subset {\rm Aut}\, (X/k)$$
via $\pi_1 \circ \pi_2$. 
By Proposition \ref{prop31}, the group $G(X,P)$ is not finitely generated.  
So, if $[{\rm Aut}\, (Y/k) : G(X, P)] < \infty$, 
then the result follows from Theorem \ref{thm21}.

In what follows, we prove $[{\rm Aut}\, (Y/k) : G(X, P)] < \infty$. 
Observe that
$$|K_{Y}| = \{E_P' + 2E_{Q}\}\,\, ,$$
where $E_P'$ is the proper transform of $E_P$ and $E_{Q}$ is the exceptional divisor of the second blow-up $Y \to Y_1$ at $Q$. Thus, for every $f\in  {\rm Aut}\, (Y/k)$, we have
$$f(E_P') = E_{P}'\,\, ,\,\, f(E_{Q}) =  E_{Q}\,\, .$$
Therefore, via $p_2$ and $p_1$, we can identify 
$${\rm Aut}\, (Y/k) = {\rm Aut}\, (Y_1/k, E_P) \cap {\rm Aut}\,(Y_1/k, Q) \subset {\rm Aut}(X/k, P)\,\, .$$ 
Let $f \in {\rm Aut}\,(Y/k)$. We regard $f \in {\rm Aut}\, (Y_1/k)$ and $f \in {\rm Aut}\, (X/k, P)$ under the identification above. Then, by Proposition \ref{prop34}, $f$ fixes $[v_1]$ and $[v_2]$ on $E_P \subset Y_1$ pointwisely.  So, $f|_{E_P} \in {\rm Aut}\,(E_P/k)$ fixes three distinct points $[v_1], [v_2], Q \in E_p(k)$ pointwisely. Thus $f|_{E_P} = id_{E_{P}}$, as $E_P \simeq \P^1$. Therefore, for $f \in {\rm Aut}\, (X/k,P)$, we have $f \in {\rm Aut}\, (Y/k)$ if and only if $f|_{E_P} = id_{E_P}$, that is, if and only if $df_P = c(f) \id_{T_{X, P}}$ for some $c(f) \in k^{\times}$. Then $f^*\omega_X = c(f)^2 \omega_X$ 
and hence 
$$G(X, P) = {\rm Ker}\, (\alpha|_{{\rm Aut}\, (Y/k)} : {\rm Aut}\, (Y/k) \to k^{\times})\,\, .$$
Here $\alpha$ is the canonical representation of ${\rm Aut}\, (X/k)$ and $\alpha|_{{\rm Aut}\, (Y/k)}$ is the restriction of $\alpha$ to ${\rm Aut}\, (Y/k)$ under ${\rm Aut}\, (Y/k) \subset {\rm Aut}\, (X/k)$. Therefore  
$$[{\rm Aut}\, (Y/k) : G(X, P)] = |{\rm Im}\, \alpha|_{{\rm Aut}\, (Y/k)}| \le |{\rm Im}\, \alpha| < \infty\,\, ,$$
by Theorem \ref{thm32}. This completes the proof of Theorem \ref{thm33}.
\end{proof}

\begin{remark}\label{rem31} Under terminologies of \cite{DO19}, what we proved here is nothing but the fact that {\it $Y$ is a core surface associated to a very special triple $(X, C, P)$ over $k$}. 
\end{remark}

\section{Proof of Corollary \ref{cor1}}\label{sect4}

In this section, we shall prove Theorem \ref{thm41}. Theorem \ref{thm1} (2) and Theorem \ref{thm41} clearly imply Corollary \ref{cor1} in Introduction. 

\begin{theorem}\label{thm41} Let $k$ be the base field as in Introduction and let $d$ be an integer such that $d \ge 3$. Choose $d-2$ integers $g_i$ ($1 \le i \le d-2$) such that 
$$2 \le g_1 < g_2 < \ldots < g_{d-2}\,\, .$$ 
Let $Y$ be a smooth projective surface in Theorem \ref{thm1} (2) and let $C_{g_i}$ be a smooth projective curve of genus $g_i$ defined over $k$. Then 
$$Y_d := Y \times C_{g_1} \times \ldots \times C_{g_{d-2}}$$
is a smooth projective variety of $\dim\, Y_{d} = d$ such that ${\rm Aut}\, (Y_d/k)$ 
is discrete and not finitely generated.
\end{theorem} 

In the rest of this section, we prove Theorem \ref{thm41}.

\begin{lemma}\label{lem41} Both ${\rm Aut}\, (C_{g_i}/k)$ and ${\rm Aut}\, (Y_d/k)$ are discrete. 
\end{lemma}

\begin{proof} By the K\"unneth formula, we have
$$H^0(Y_d, T_{Y_d}) = H^0(Y, T_Y) \oplus H^0(C_{g_1}, T_{C_{g_1}}) \oplus 
\ldots \oplus  H^0(C_{g_{d-2}}, T_{C_{g_{d-2}}})\,\, .$$
As $\deg\, T_{C_{g_i}} = 2-2g_i < 0$, it follows that $H^0(C_{g_{i}}, T_{C_{g_{i}}}) = 0$. By our choice of $Y$, we have $H^0(Y, T_Y) = 0$ as well (cf. Remark \ref{rem1}). Hence $H^0(Y_d, T_{Y_d}) = 0$ and we are done.
\end{proof}

\begin{remark}\label{rem41} There is a smooth projective surface $S$ of general type with non-zero regular global vector field over $k$ (\cite{La83}). In particular, unlike in characteristic zero, ${\rm Aut}\, (V/k)$, and hence ${\rm Aut}\, ((Y \times V)/k)$, could be non-discrete even if $V$ is a smooth projective variety of general type. 
\end{remark}

Set 
$$Z := C_{g_1} \times \ldots \times C_{g_{d-2}}\,\, .$$

\begin{lemma}\label{lem42} One has  
$${\rm Aut}\, (Z/k) = {\rm Aut}\, (C_{g_1}/k) \times \ldots \times {\rm Aut}\, (C_{g_{d-2}}/k)$$
under the natural inclusion of the right hand side into the left hand side.  
\end{lemma}

\begin{proof} We prove the equality by the induction on $d-2$. If $d-2 = 1$, then the result is clear. Now assume $d-2 \ge 2$. Set 
$$Z' := C_{g_2} \times \ldots \times C_{g_{d-2}}\,\, .$$
Then $Z = C_{g_1} \times Z'$. We denote any closed point of $Z$ as $(x, t)$ where $x \in C_{g_{1}}$ and $t \in Z'$. 

Notice that genus does not change under any inseparable morphism. Thus, there is no non-constant morphism from $C_{g_i}$ to $C_{g_{j}}$ whenever $i < j$, that is, whenever $g_i < g_j$ (See eg. \cite[Chap IV, Sect 4.2]{Ha77}). Hence if $C \subset Z$ is isomorphic to $C_{g_1}$, then $C$ is a fiber of the projection to the second factor: 
$$\pi : Z = C_{g_1} \times Z' \to Z'\,\, ;\,\, (x, t) \mapsto t\,\, .$$
Hence ${\rm Aut}\, (Z/k)$ preserves $\pi$. It follows that any $F \in {\rm Aut}\, (Z/k)$, which is discrete, is of the form
$$F(x, t) = (f_t(x), f(t))$$
where $f \in {\rm Aut}\, (Z')$ and $f_t \in {\rm Aut}\, (C_{g_1}/k)$ parametrized by $t \in Z'$. As ${\rm Aut}\, (C_{g_1}/k)$ is discrete, it follows that $f_t$ does not depend on $t$. Thus 
$${\rm Aut}\, (Z/k) = {\rm Aut}\, (C_{g_1}/k) \times {\rm Aut}\, (Z'/k)\,\, ,$$
and the result follows from the induction on $d-2$. 
\end{proof}

\begin{lemma}\label{lem43} ${\rm Aut}\, (Z/k)$ is a finite group.  
\end{lemma}

\begin{proof} This follows from Lemma \ref{lem42} and the fact that ${\rm Aut}\, (C_{g_i}/k)$ is a finite group. The finiteness of ${\rm Aut}\, (C_{g_i}/k)$ can be shown as follows. By considering pluricanonical morphisms of $C_{g_i}$, one can regard ${\rm Aut}\, (C_{g_i}/k)$ as a Zariski closed subscheme of ${\rm PGL}\,(N, k)$ for some positive integer $N$. As ${\rm Aut}\, (C_{g_i}/k)$ is discrete (Lemma \ref{lem41}), ${\rm Aut}\, (C_{g_i}/k)$ is then a reduced Zariski closed subscheme of dimension $0$ of ${\rm PGL}\,(N, k)$. As ${\rm PGL}\,(N, k)$ is noetherian, it follows that $|{\rm Aut}\, (C_{g_i}/k) | < \infty$ as claimed. 
\end{proof}

\begin{lemma}\label{lem44} One has 
$${\rm Aut}\, (Y_d/k) = {\rm Aut}\, (Y/k) \times {\rm Aut}\, (Z/k)$$
under the natural inclusion of the right hand side into the left hand side.    
\end{lemma}

\begin{proof} We have $Y_d = Y \times Z$. As $Y$ is birational to a K3 surface, the $m-$th canonical map $\Phi_{|mK_{Y_d}|}$ of $Y_d$ with sufficiently large $m$ is nothing but the projection from $Y_d$ to the second factor:
$$p: Y_d := Y \times Z \to Z\,\, .$$
From now, our proof is very close to the proof of Lemma \ref{lem42}. We denote any closed point of $Y_d$ as $(y, z)$ where $y \in Y$ and $z \in Z$. 
As ${\rm Aut}\, (Y/k)$ preserves the $m-$th canonical map, it follows that any $G \in {\rm Aut}\, (Y_{d-2}/k)$, which is discrete, is of the form
$$G(y, z) = (g_z(y), g(z))$$
where $g \in {\rm Aut}\, (Z)$ and $g_z \in {\rm Aut}\, (Y/k)$ parametrized by $z \in Z$. As ${\rm Aut}\, (Y/k)$ is discrete, it follows that $g_z$ does not depend on $z$. Thus 
$${\rm Aut}\, (Y_d/k) = {\rm Aut}\, (Y/k) \times {\rm Aut}\, (Z/k)\,\, ,$$
as claimed.
\end{proof} 

By Lemma \ref{lem41}, ${\rm Aut}\, (Y_d/k)$ is discrete. By Lemma \ref{lem43} and Lemma \ref{lem44}, ${\rm Aut}\, (Y_d/k)$ has a finite index subgroup which is isomorphic to ${\rm Aut}\, (Y/k)$. By our choice of $Y$, the group ${\rm Aut}\, (Y/k)$ is not finitely generated (Theorem \ref{thm1} (2)). Hence by Theorem \ref{thm21}, ${\rm Aut}\, (Y_d/k)$ is not finitely generated as well. This completes the proof of Theorem \ref{thm41}. 

\end{document}